\documentclass[11pt]{amsart}
\usepackage{amsfonts,amscd,latexsym,amsmath,amssymb,enumerate,verbatim,amsthm,epsfig,alltt,enumitem,color,bbm}

\newcommand{\Homeo}{\mathrm{Homeo}}
\newcommand{\Aut}{\mathrm{Aut}}
\newcommand{\Iso}{\mathrm{Iso}}
\newcommand{\LIso}{\mathrm{LIso}}

\newcommand{\Rea}{\mathbb{R}}
\newcommand{\Nat}{\mathbb{N}}
\newcommand{\Int}{\mathbb{Z}}

\newcommand{\PP}{\mathcal{P}}

\newcommand{\XX}{\mathcal{X}}
\newcommand{\YY}{\mathcal{Y}}

\newcommand{\DD}{\mathcal{D}}
\newcommand{\Act}{\mathrm{Act}}

\newcommand{\vertiii}[1]{{\left\vert\kern-0.25ex\left\vert\kern-0.25ex\left\vert #1 
		\right\vert\kern-0.25ex\right\vert\kern-0.25ex\right\vert}}

\theoremstyle{plain}

\newtheorem{theorem}{\bf Theorem}[section]
\newtheorem{lemma}[theorem]{\bf Lemma}
\newtheorem{proposition}[theorem]{\bf Proposition}
\newtheorem{corollary}[theorem]{\bf Corollary}
\newtheorem{fact}[theorem]{\bf Fact}

\theoremstyle{definition}
\newtheorem{definition}[theorem]{\bf Definition}

\newtheorem{remark}[theorem]{\bf Remark}

\newtheorem{question}[theorem]{\bf Question}

\title{Invariant strictly convex renormings}
\author{Michal Doucha}
\address{Institute of Mathematics, Czech Academy of Sciences, Žitná 25, 115 67 Praha 1, Czechia}
\email{doucha@math.cas.cz}
\urladdr{https://users.math.cas.cz/~doucha/}

\thanks{The author was supported by the GA\v{C}R project 25-15366S and by the Czech Academy of Sciences (RVO 67985840).}
\begin{document}
	\maketitle
	
\begin{abstract}
Motivated by the question of Mikael de la Salle, we investigate the problem of the existence of equivalent strictly convex norms on Banach spaces that are invariant with respect to an action of a group by linear isometries. We develop various tools for constructing such norms and prove several preservation results. We also answer positively a question of Antunes, Ferenczi, Grivaux and Rosendal whether there is a strictly convex renorming of $c$ invariant with respect to its full linear isometry group.

Finally, we specialize to the spaces $L_1[0,1]$ and $C(K)$, where $K$ is compact Hausdorff, and indicate that amenability of groups plays a role in this problem.
\end{abstract}

\section{Introduction}
During the problem session of the 2018 CIRM conference `Non Linear Functional Analysis', Mikael de la Salle proposed to study the following problem. Let $G$ be a (topological) group acting by linear isometries on a Banach space $X$. Under which conditions, on the Banach space, the group, and the action, does there exist an equivalent strictly convex norm on $X$ that is invariant with respect to the action?

This problem had been considered in some form before. It appeared in the work of Lancien \cite{Lan93} in the context of locally uniformly convex renorming. Invariant locally uniformly convex renormings were then more systematically studied by Ferenczi and Rosendal in \cite{FR11,FR13} and Antunes, Ferenczi, Grivaux, Rosendal in \cite{AFGR19}, where the main emphasis was on the existence of a renorming that is invariant with respect to the full linear isometry group of the Banach space. The shift of the emphasis in de la Salle's question on more general actions of possibly simpler groups proves to be meaningful. Indeed, we will show there are Banach spaces for which there are no stricly convex renormings with respect to the full linear isometry groups, yet we can say something about individual actions of smaller groups.

The renorming theory is a well-established and well-developed subject within the Banach space theory, see the monographs \cite{DGZ-book,GMZ-book}. Nevertheless, the `invariant part' of the theory is yet to be properly explored as the type of problems investigated in this work seen in the classical theory belong to the ancient history of the subject. The fact that every separable Banach space admits a strictly convex renorming was proved by Clarkson in \cite{Cla36} and first examples of Banach spaces without strictly convex renorming appear in Day's paper \cite{Day55}. In the invariant theory, these problems are already a challenge and the research in this topic connects the classical renorming theory of Banach spaces with the theory of group actions on Banach spaces and generally measurable and topological dynamics of groups.

\subsection{Main results} Here we summarize the main results of this paper.
\begin{itemize}
	\item We provide a characterization of Banach spaces equipped with a continuous action of a topological group by linear isometries admitting an invariant strictly convex renormings, including a special version for separable Banach spaces (see Propositions~\ref{prop:SCrenorming1} and~\ref{prop:SCrenorming2}, and Corollary~\ref{cor:SCrenorming}).
	\item We show that Banach spaces admitting a strictly convex renormings invariant under their full linear isometry groups are closed under $\ell_1$ and $c_0$-sums, under some natural mild conditions (see Theorems~\ref{thm:l1sumspreservation} and~\ref{thm:closurec0sums}).
	\item We answer positively \cite[Question 8.1]{AFGR19} by showing that the Banach space $c$ admits a strictly convex renorming invariant with respect to its full linear isometry group (see Theorem~\ref{thm:cGSCR}).
	\item We show that despite the fact there is an action of $F_2$, the free group on $2$ generators, on $L_1[0,1]$ by lattice isometries without an invariant strictly convex renorming, every countable group admits a free action on $L_1[0,1]$ by lattice isometries that admit an invariant strictly convex renorming. When the group is amenable, the set of such actions is dense in the space of actions (see Proposition~\ref{prop:nastyF2}, Corollary~\ref{cor:actiononL1}, and Theorem~\ref{thm:densityforL1}).
	\item We provide conditions under which actions of amenable groups on $C(K)$ spaces admit invariant strictly convex renormings and we establish a `quasi-dichotomy': A generic action by linear isometries of a countable locally finite group on $C(2^\Nat)$ admits an invariant strictly convex renorming, yet a generic action by positive linear isometries of a group containing $F_2$ on $C(2^\Nat)$ does not admit an invariant strictly convex renorming (see Theorems~\ref{thm:cptspaces}, {}~\ref{thm:locfinite}, and~\ref{thm:haveF2}).
\end{itemize}

\subsection{Organization of the paper} The next Section~\ref{sec:prelim} provides a quick introduction to all the main notions appearing in the paper, including convex norms, group actions on Banach spaces and spaces of actions, amenable groups, and descriptions of the full linear isometry groups of $L_1[0,1]$ and $C(K)$ spaces.

Section~\ref{sec:general} provides characterizations of when a Banach space equipped with a group action by linear isometries admits an invariant strictly convex renormings - these will be applied in further sections. Section~\ref{sec:fulliso} focuses on preservations results and answers \cite[Question 8.1]{AFGR19}. Sections~\ref{sec:L1}, resp.~\ref{sec:C(K)} investigate invariant renormings of $L_1[0,1]$, resp. of $C(K)$. Finally, Section~\ref{sec:questions} presents several questions whose answers could push the subject further.
\section{Notation and preliminaries}\label{sec:prelim}
\noindent\textbf{Agreement.} All Banach spaces in this paper are considered to be real, although the complex case can be treated in virtually every argument here the same or very similar. Topological groups are always considered to be Hausdorff.\medskip

Given a Banach space $X$, we shall denote by $\Iso(X)$ the group of all linear isometries of $X$ equipped with the strong operator topology. When $X$ has an additional lattice structure, we shall denote by $\LIso(X)$ the subgroup of all lattice (or positive) linear isometries with the inherited topology, i.e. again the strong operator topology.\medskip

\subsection{Group actions.} If $G$ is a topological group and $X$ is a topological space, a \emph{(continuous) action} $\alpha$ of $G$ on $X$ is a (continuous) map $\alpha:G\times X\to X$ satisfying
\begin{itemize}
	\item $\alpha(gh,x)=\alpha\big(g,\alpha(h,x)\big)$, for all $g,h\in X$ and $x\in X$;
	\item $\alpha(1_G,x)=x$ for all $x\in X$.
\end{itemize}
In the sequel all actions are assumed to be continuous, even without explicitly mentioning the continuity. The topological spaces on which groups will act in this document are mainly Banach spaces, in which case the actions will be by linear isometries, or compact Hausdorff spaces. As an exception to this case, we will consider actions on measure spaces in Section~\ref{sec:L1}.

We remark a basic fact that given a topological group $G$ and a continuous action $\alpha:G\times X\to X$ of $G$ on a Banach space $X$ by linear isometries, we can identify $\alpha$ with a continuous homomorphism, by abusing the notation denoted by the same symbol, $\alpha:G\to\Iso(X)$ defined by $\alpha(g)(x):=\alpha(g,x)$, for $g\in G$ and $x\in X$. Conversely, every continuous homomophism $\alpha:G\to\Iso(X)$ defined a continuous action $\alpha:G\times X\to X$ by linear isometries defined by $\alpha(g,x):=\alpha(g)(x)$, for $g\in G$ and $x\in X$.\medskip

A Banach space equipped with an action of a topological group $G$ by linear isometries will be sometimes called a \emph{$G$-Banach space}. The natural morphism between $G$-Banach spaces is the following.

\begin{definition}
Let $G$ be a topological group and $X$ and $Y$ be $G$-Banach spaces. We call a linear operator $\phi:X\to Y$ \emph{$G$-equivariant}, or just \emph{equivariant}, if $\phi(g\cdot x)=g\cdot \phi(x)$ for all $g\in G$ and $x\in X$.
\end{definition}

\subsubsection{Spaces of actions}
Later in the paper there will be situations where we cannot, or it is not possible, to prove that given a topological group $G$ and a Banach space $X$, every/no continuous action of $G$ on $X$ admits a $G$-invariant strictly convex renorming. However, it will be possible to say that a `majority' of actions do. To formalize this, we need to introduce the space of actions. We restrict just to the case of countable (discrete) groups. The setting can be considered in bigger generality though, e.g. for locally compact second countable groups.

Fix a countable group $G$ and a Banach space $X$. The goal is to define a topological space of all actions of $G$ on $X$ by linear (or positive linear if $X$ is a Banach lattice) isometries. Since actions of $G$ on $X$ by linear isometries can be identified with homomorphisms from $G$ into $\Iso(X)$, we can topologize the space of all actions of $G$ on $X$ by linear isometries, resp. positive linear isometries if $X$ is a Banach lattice, denoted by $\Act_G(X)$, resp. $\Act_G^+(X)$, by seeing it as a closed subset of $\Iso(X)^G$, resp. $\LIso(X)^G$, equipped with the product topology. This gives a completely metrizable, in particular Baire, topology on $\Act_G(X)$ and $\Act_G^+(X)$. If  $X$ is separable, then $\Act_G(X)$ and $\Act_G^+(X)$ are Polish.
\subsection{Convex norms on Banach spaces}
Here we recall several notions of convexity of norms on Banach spaces. In particular, we recall the definition of a strictly convex norm.

\begin{definition}
Let $(X,\|\cdot\|)$ be a Banach space. Recall that $\|\cdot\|$ is
\begin{itemize}
	\item \emph{strictly convex} if for every $x\neq y\in X$ with $\|x\|=\|y\|$ necessarily $\|\frac{x+y}{2}\|<\|x\|$.
	
	\item \emph{locally uniformly convex (LUC)} if for every $x\in X$ and a sequence $(x_n)_n\subseteq X$ we have $\|x\|=\lim_{n\to\infty} \|x_n\|=\lim_{n\to\infty} \|\frac{x+x_n}{2}\|$ if and only if $\lim_{n\to\infty} \|x_n-x\|=0$.
	
	\item \emph{uniformly convex} if for every $0<\varepsilon\leq 2$ there is $\delta>0$ such that for $x,y\in S_X$, $\|x-y\|\geq \varepsilon$ implies $\|\frac{x+y}{2}\|<1-\delta$.
\end{itemize}

Clearly, uniformly convex norm is locally uniformly convex and locally uniformly convex norm is strictly convex.
\end{definition}
Standard examples of Banach spaces with uniformly convex norms are $L_p(\Omega,\mu)$, where $p\in(1,\infty)$ and $(\Omega,\mu)$ is any measure space. On the other hand, $L_p(\Omega,\mu)$, for $p\in\{1,\infty\}$, and unless one-dimensional, is not strictly convex.

The following fact is well known and will be used without explicit reference throughout the paper.
\begin{fact}
Let $(X_n)_{n\in\Nat}$ be a sequence of strictly convex Banach spaces. Then for $p\in(1,\infty)$, their $\ell_p$-sum $\bigoplus_{\ell_p} X_n$ is strictly convex.
\end{fact}
Classical results of James and Enflo, see \cite{Jam1,Jam2,Jam3,Enf72}, state that a Banach space admits an equivalent uniformly convex norm if and only if it is superreflexive. On the other hand, Kadec showed (see \cite{Kad59}) that every separable Banach space admits an equivalent locally uniformly convex norm. There are plethora of examples of non-separable Banach spaces that do not admit an equivalent strictly convex norm; e.g., Bourgain in \cite{Bour80} showed that $\ell_\infty/c_0$ is an example. Note that $\ell_\infty$ is a non-separable Banach space whose norm is not strictly convex, yet it does admit an equivalent strictly convex norm (this is the case for all spaces with a separable predual).

\subsubsection{Invariant convex norms.} Let $G$ be a topological group and $X$ be a $G$-Banach space, i.e. the group $G$ acts on $X$ by linear isometries with respect to the original defining norm of $X$. Let $\vertiii{\cdot}$ be any equivalent norm on $X$. Say that $\vertiii{\cdot}$ is \emph{$G$-invariant} if the action of $G$ is by linear isometries even with respect to $\vertiii{\cdot}$, i.e. $\vertiii{g\cdot x}=\vertiii{x}$, for all $x\in X$ and $g\in G$.\medskip

\noindent\textbf{Goal.} It is the main topic of the paper to determine for which groups $G$, Banach spaces $X$, and actions of $G$ on $X$, there are equivalent convex norms on $X$ that are $G$-invariant.

Let us start by recalling that this problem has a straightforward solution for uniformly convex norms.
\begin{fact}[Bader, Furman, Gelander, Monod {\cite[Proposition 2.3]{BFGM07}}]
Let $G$ be a topological group and $X$ be a superreflexive $G$-space. Then $X$ admits a $G$-invariant uniformly convex norm.
\end{fact}

For locally uniformly convex norms this problem is already a challenge. The most powerful and general result in this direction is the following result of Lancien.
\begin{theorem}[Lancien {\cite[Theorem 2.1 and Remark after Lemma 2.5]{Lan93}}]\label{thm:Lancien}
Let $G$ be a topological group and $X$ be a separable $G$-space with the Radon-Nikod\'ym property. Then $X$ admits an equivalent $G$-invariant locally uniformly convex norm.
\end{theorem}
For further research on invariant locally uniformly convex norms we refer to papers of Ferenczi and Rosendal \cite{FR11,FR13} and Antunes, Ferenczi, Grivaux, Rosendal \cite{AFGR19}. Particular results from these papers will be mentioned in appropriate places within this document.

\subsection{Amenable groups}
Amenability of groups will be shown to play a role in finding equivalent invariant strictly convex norms. In reality, the importance of amenability is likely much bigger than it is revealed in this paper and ought to be investigated further - although, as pointed out to us by C. Rosendal, amenability alone cannot guarantee the existence of an invariant strictly convex norm (see Section~\ref{sec:questions} for details).

We therefore provide a basic information about this concept.

\begin{definition}
A topological group $G$ is called \emph{amenable} if every continuous action of $G$ on a compact Hausdorff space $K$ admits an invariant regular Borel probability measure.
\end{definition}
We note that a measure $\mu$ on a topological space $X$ on which $G$ acts is \emph{invariant} if for every $g\in G$ and the homeomorphism of $X$ it induces by the action, denoted still by $g$, the pushforward measure $g_*\mu$ is equal to $\mu$.\smallskip

Examples of amenable groups include all compact groups - in particular finite groups, abelian groups - in particular $\Int$. Prototypical examples of non-amenable groups are non-abelian free groups.\medskip

Powerful tools involving amenable groups that will prove instrumental in finding invariant strictly convex norms on $L_1[0,1]$ in this paper include the Ornstein-Weiss Rokhlin's lemma for amenable groups \cite{OrWei87} as well as the tiling machinery of Downarowicz, Huczek and Zhang \cite{DoHuZh19}. Amenability will also appear in the context of invariant strictly convex norms on $C(2^\Nat)$ applying the recent results from \cite{DoMeTs25}.

\subsection{Banach-Stone and Banach-Lamperti theorems}
The Banach-Stone and Banach-Lamperti theorems completely describe (positive) linear isometries of spaces of the form $C(K)$ and $L_p(\Omega,\mu)$, $p\in\left[1,\infty\right)\setminus\{2\}$, where $K$ is a compact Haudorff space and $(\Omega,\mu)$ is a measure space. As such, they will be invaluable instruments when investigating group actions on these Banach spaces.

Before we recall these theorems, we digress a bit to recall the semi-direct construction in group theory as the linear isometry groups of $C(K)$ and $L_p(\Omega,\mu)$ spaces, $\in\left[1,\infty\right)\setminus\{2\}$, are natural semi-direct product of groups.

\subsubsection{Semi-direct product}
Let $G$ and $H$ be topological groups and denote by $\Aut(H)$ the group of all automorphisms of $H$ with the compact-open topology. Suppose that $G$ admits a continuous action on $H$, i.e. there is a continuous homomorphism $\xi$ from $G$ into $\Aut(H)$. Then one can form a semi-direct product $G\ltimes H$, which as a topological space is homeomorphic to the direct product $G\times H$ and the group operation is defined by \[(g_1,h_1)\cdot (g_2,h_2):=(g_1g_2,\xi(g_2^{-1})h_1h_2),\quad g_1,g_2\in G,\; h_1,h_2\in H.\] Notice that both $G$ and $H$ topologically and algebraically embed into $G\ltimes H$ and we identify them with their canonical copies inside $G\ltimes H$. $H$ is then a closed normal subgroup of $G\ltimes H$ and $(G\ltimes H)/H=G$.

The most interesting examples for us are described below.
\subsubsection{Banach-Stone theorem}
Given a compact Hausdorff space $K$, denote by $\Homeo(K)$ the group of all homeomorphisms of $K$ equipped with the uniform convergence topology. Notice that given $\phi\in\Homeo(K)$ there is a positive linear isometry $T_\phi\in \LIso(C(K))$ defined by \[T_\phi(f)(x):=f\big(\phi^{-1}(x)\big),\quad f\in C(K),\; x\in K.\] The map $\phi\in\Homeo(K)\to T_\phi\in\LIso(C(K))\subseteq\Iso(C(K))$ is a continuous injective homomorphism.

Denote also by $C(K,\{-1,1\})$ the space of all continuous functions from $K$ into $\{-1,1\}$ with the compact-open topology, which form an abelian group under pointwise multiplication. For each $\phi\in C(K,\{-1,1\})$, there is $T_\phi\in\Iso(C(K))$ defined by \[T_\phi(f)(x):=\phi(x)f(x),\quad f\in C(K),\; x\in K.\] The map $\phi\to T_\phi$ is again a continuous injective homomorphism. The natural action $\xi$ of $\Homeo(K)$ on $C(K,\{-1,1\})$ defined by \[\xi(\phi,\psi)(x):=\psi\big(\phi^{-1}(x)\big),\quad \phi\in\Homeo(K),\;\psi\in C(K,\{-1,1\}),\; x\in K\] lets us define the semi-direct product $\Homeo(K)\ltimes C(K,\{-1,1\})$. The Banach-Stone theorem, see e.g. \cite[Theorem 2.1.1]{FleJabook1}, can be then interpreted that $\Iso(C(K))$ is equal (or topologically isomorphic) to $\Homeo(K)\ltimes C(K,\{-1,1\})$.

It also follows that one can then identify $\LIso(C(K))$ and $\Homeo(K)$.

\subsubsection{Banach-Lamperti theorem}
The Banach-Lamperti theorem applies to $L_p(\Omega,\mu)$, $p\in\left[1,\infty\right)\setminus\{2\}$, where $(\Omega,\mu)$ is a measure space, analogously as the Banach-Stone theorem applies to $C(K)$. For simplicity, we shall mention here only the case when $(\Omega,\mu)$ is $([0,1],\lambda)$, since that is all we need in this paper. This case was proved already by Banach in \cite[Chapter 11]{Ban55}.

Let $\phi:[0,1]\to[0,1]$ be a \emph{measure-class preserving} (also called \emph{non-singular}) measurable bijection. That is, $\phi$ preserves the equivalence class of $\lambda$, i.e. $\phi_*\lambda\ll\lambda$ and $\lambda\ll\phi_*\lambda$, where $\phi_*\lambda$ is the pushforward of $\lambda$ via $\phi$ and `$\ll$' denotes the absolute continuity between measures. There is a positive linear isometry $T_\phi:L_1[0,1]\to L_1[0,1]$ defined by \[T_\phi(f)(x):=\frac{d\phi_*\lambda}{d\lambda}(x)f\big(\phi^{-1}(x)\big),\quad f\in L_1[0,1],\;x\in [0,1],\] where $\frac{d\phi_*\lambda}{d\lambda}$ is the Radon-Nikod\'ym derivative of $\phi_*\lambda$ with respect to $\lambda$. Denote by $\Aut([0,1],\lambda)$ the group of all measure-class preserving measurable bijections modulo the congruence of two maps being congruent when they coincide up to measure zero. As in the case of elements of $L_1[0,1]$, we shall usually work with concrete representatives rather than equivalence classes. We equip $\Aut([0,1],\lambda)$ with the (group) topology generated by the subbasic open sets of the form $\big\{\psi\in\Aut([0,1],\lambda)\colon \lambda(\psi[A]\triangle\phi[A])<\varepsilon\big\}$, where $\phi\in\Aut([0,1],\lambda)$, $A\subseteq [0,1]$ is Lebesgue measurable, and $\varepsilon>0$. The map $\phi\in\Aut([0,1],\lambda)\to T_\phi\in\LIso(L_1[0,1])$ is an injective continuous homomorphism.

We also consider the group $M([0,1],\{-1,1\})$ of all measurable maps from $[0,1]$ into $\{-1,1\}$, modulo the congruence of two maps being congruent when they coincide up to measure zero, with pointwise multiplication and equipped with the coarsest topology that makes the map \[(\phi,\psi)\in M([0,1],\{-1,1\})^2\to \lambda\big(x\in[0,1]\colon \phi(x)\neq\psi(x)\big)\] continuous. There is again a natural continuous injective group homomorphism $\phi\in M([0,1],\{-1,1\})\to T_\phi$, where $T_\phi$ is defined by \[T_\phi(f)(x):=\phi(x)f(x),\quad f\in L_1[0,1],\;x\in[0,1].\]

As in the Banach-Stone case, there is a natural action $\xi$ of $\Aut([0,1],\lambda)$ on $M([0,1],\{-1,1\})$ defined by \[\xi(\phi,\psi)(x):=\psi\big(\phi^{-1}(x)\big),\quad \phi\in\Aut([0,1],\lambda),\;\psi\in M([0,1],\{-1,1\}),\; x\in [0,1]\] which again lets us define the semi-direct product $\Aut([0,1],\lambda)\ltimes M([0,1],\{-1,1\})$. The Banach-Lamperti theorem, see e.g. \cite[Theorem 3.2.5]{FleJabook1}, can be then interpreted that $\Iso(L_1[0,1])$ is equal (or topologically isomorphic) to $\Aut([0,1],\lambda)\ltimes M([0,1],\{-1,1\})$. It then also follows that we can identify $\LIso(L_1[0,1])$ and $\Aut([0,1],\lambda)$.

\section{General results}\label{sec:general}
Here we provide a characterization of $G$-Banach spaces admitting an equivalent $G$-invariant strictly convex norm. The characterization is rather simple, yet quite useful, and along with Proposition~\ref{prop:sup=max} from Section~\ref{sec:fulliso} will be our main tool in subsequent sections.
\begin{proposition}\label{prop:SCrenorming1}
Let $X$ be a $G$-Banach space. Then there exists a strictly convex $G$-invariant renorming if and only if there exist a strictly convex $G$-Banach space $Y$ and a bounded injective $G$-equivariant linear operator from $X$ to $Y$.
\end{proposition}
\begin{proof}
Suppose there exists an equivalent $G$-invariant strictly convex norm $\|\cdot\|'$ on $X$. Then the identity map $\mathrm{id}:(X,\|\cdot\|_X)\to (X,\|\cdot\|')$ is the desired bounded injective $G$-equivariant linear operator.

Conversely, if there exist a strictly convex $G$-Banach space $Y$ and a bounded injective $G$-equivariant linear operator $T:X\to Y$ then \[\|x\|_0:=\|T(x)\|_Y\] defines a $G$-invariant strictly convex norm on $X$ staisfying $\|\cdot\|_0\leq \|T\|\|\cdot\|_X$ and $\|\cdot\|':=\|\cdot\|_X+\|\cdot\|_0$ is clearly equivalent $G$-invariant strictly convex norm on $X$, as a sum of two $G$-invariant norms on $X$ where one of them is equivalent and the other is strictly convex and bounded by a scalar multiple of the original norm.
\end{proof}
As a corollary we get that when there exists an equivalent invariant strictly convex norm, then it can be chosen arbitrarily close to the original norm. This is a quite useful observation that will be applied in the latter sections.
\begin{corollary}\label{cor:SCrenorming}
Let $X$ be a $G$-Banach space. If there exists a strictly convex $G$-invariant renorming on $X$ then there exists a $G$-invariant strictly convex norm $\vertiii{\cdot}$ such that $(1-\varepsilon)\|\cdot\|_X\leq \vertiii{\cdot}\leq (1+\varepsilon)\|\cdot\|_X$, for every $\varepsilon>0$.
\end{corollary}
\begin{proof}
Suppose there exists a $G$-invariant strictly convex renorming and fix $\varepsilon>0$. Then by Proposition~\ref{prop:SCrenorming1} there are a strictly convex $G$-space $Y$ and a bounded injective $G$-equivariant linear operator $\phi:X\to Y$. We set \[\vertiii{x}:=\|x\|_X+\frac{\varepsilon \phi(x)}{\|\phi\|},\quad x\in X.\] Clearly, $\vertiii{\cdot}$ is strictly convex and $G$-invariant satisfying $(1-\varepsilon)\|\cdot\|_X\leq \vertiii{\cdot}\leq (1+\varepsilon)\|\cdot\|_X$.
\end{proof}

The next proposition is a variant of Proposition~\ref{prop:SCrenorming1} for separable Banach spaces which are of the main interest in this paper.
\begin{proposition}\label{prop:SCrenorming2}
Let $X$ be a separable Banach space and $G$ be a topological group acting continuously on $X$ by isometries. There exists a strictly convex $G$-invariant renorming of $X$ if and only if there are a countable dense sequence $(x_n)_{n\in\Nat}\subseteq X$, a constant $C>0$, and linear $G$-equivariant maps $T_n:X\to Z_n$, for $n\in\Nat$, such that
\begin{itemize}
	\item $Z_n$ is a strictly convex Banach space equipped with an action of $G$ by isometries,
	\item  $\|T_n\|\leq 1$,
	\item $\|T_n(x_n)\|\geq C\|x_n\|$.
\end{itemize}
\end{proposition}
\begin{proof}
The forward direction is clear and proved analogously as the forward direction of Proposition~\ref{prop:SCrenorming1}.

Conversely, set \[Z:=\bigoplus_{\ell_2} Z_n,\] which is strictly convex as an $\ell_2$-sum of strictly convex spaces. Further, we set $T:X\to Z$ by \[x\to \big(\frac{T_n(x)}{2^n}\big)_{n\in\Nat}\] which is clearly a $G$-equivariant operator of norm at most $1$. We claim that $T$ is injective. Indeed, pick any non-zero $x\in X$ and then choose $n\in\Nat$ so that $\|x-x_n\|<C\|x_n\|$. Then $\|T_n(x)\|\geq \|T_n(x_n)\|-\|T_n(x-x_n)\|>0$, thus $T_n(x)\neq 0$ and in particular $T(x)\neq 0$.

The statement then follows from Proposition~\ref{prop:SCrenorming1}.
\end{proof}

\section{Actions of the full isometry groups}\label{sec:fulliso}
Given a Banach space $X$, the strongest requirement in our setting we can ask for is whether there is an invariant strictly convex renorming with respect to any action of any group. This problem was considered for locally uniformly convex norms by Ferenczi and Rosendal, resp. with their coauthors, in \cite{FR11,FR13}, resp. in \cite{AFGR19}. We have already mentioned Theorem~\ref{thm:Lancien} of Lancien in this direction.

Here we show that more can be proved for the weaker condition of admitting an invariant strictly convex renorming and that this property of Banach spaces has several closure properties. We also answer positively an open question of Antunes, Ferenczi, Grivaux and Rosendal \cite[Question 8.1]{AFGR19}.

First we give a name to the condition we are after.
\begin{definition}
Say that a Banach space $X$ has \emph{group-invariant strictly convex renormable property (GSCR)} if for every group $G$ and every continuous action $G\curvearrowright X$ by isometries there is a $G$-invariant strictly convex renorming of $X$.
\end{definition}

\begin{remark}
A Banach space $X$ has GSCR if there is a $G$-invariant strictly convex renorming of $X$ for $G=\Iso(X)$ with the canonical action.
\end{remark}
The next proposition is along with the results from Section~\ref{sec:general} the only tool we have for finding invariant strictly convex renormings.
\begin{proposition}\label{prop:sup=max}
	Let $G$ be a topological group acting continuously by linear isometries on a Banach space $X$. Suppose that $X$ admits a strictly convex norm $|\cdot|$. If in the definition \[\vertiii{x}:=\sup_{g\in G}|gx|,\quad x\in X\] the supremum is attained for every $x\in X$, then $\vertiii{\cdot}$ is a $G$-invariant strictly convex norm.
\end{proposition}
\begin{proof}
	Suppose the supremum is attained. By definition, the norm $\vertiii{\cdot}$ is $G$-invariant so we just check it is strictly convex. Pick $x\neq y\in X$ such that $\vertiii{x}=\vertiii{y}=1$ and suppose that $\vertiii{\frac{x+y}{2}}=1$. Let $g\in G$ be such that $|\frac{gx+gy}{2}|=1$. Then, since $|\cdot|$ is strictly convex \[\vertiii{x}+\vertiii{y}\geq |gx|+|gy|>2,\]a contradiction.
\end{proof}
As an immediate corollary we recover the following results from \cite{AFGR19}.
\begin{corollary}\label{cor:c0}\emph{}\\\begin{enumerate}
		\item For every Banach space admitting a strictly convex norm and a compact $G$ acting on $X$ by linear isometries, there is a $G$-invariant strictly convex norm on $X$ (see \cite[Proposition 2.1]{AFGR19}) .
		\item The Banach space $c_0$ has GSCR (see \cite[Example 4.2]{AFGR19}).
	\end{enumerate}
\end{corollary}
\begin{proof}
The former is clear since if $G$ is compact the supremum in the statement of Proposition~\ref{prop:sup=max} is attained by compactness.

For $c_0$, consider the following strictly convex norm. For $x\in c_0$ set \[|x|:=\|x\|_{c_0}+\Big(\sum_{i=1}^\infty \frac{x_i^2}{4^i}\Big)^{1/2}.\] Let $G=\Iso(c_0)$ and set \[\vertiii{x}:=\sup_{g\in G}|gx|,\quad x\in X.\] We show that the supremum is attained, thus Proposition~\ref{prop:sup=max} applies. Pick $x\in c_0$ and let $h\in G$ acts as a permutation of $\Nat$ that rearranges $x$ so that $hx$ is a decreasing sequence; that is, for every $i\in\Nat$, $|hx_i|\geq |hx_{i+1}|$. It is then clear that $\sup_{g\in G} |gx|=|hx|$. This finishes the proof.
\end{proof}

Another corollary shows that the class of Banach spaces with GSCR is closed under $c_0$-sums - under some conditions on the spaces in the sum.

We first need the following proposition.

\begin{proposition}\label{prop:iso-of-l1-c0-sums}
	Let $(X_n)_{n\in\Nat}$ be a sequence of separable Banach spaces.
	\begin{enumerate}
		\item If none of them can be decomposed as an $\ell_1$-sum of two proper subspaces, then every surjective linear isometry of $\bigoplus_{\ell_1} X_n$ is given by surjective linear isometries of individual summands and permutation that exchange isometric summands.
		\item If none of them can be decomposed as an $\ell_\infty$-sum of two proper subspaces, then every surjective linear isometry of $\bigoplus_{c_0} X_n$ is given by surjective linear isometries of individual summands and permutation that exchange isometric summands.
	\end{enumerate}

	That is, in both cases, let $I$ be a finite or countably infinite set of isometry class of spaces $(X_n)_{n\in\Nat}$. For each $i\in I$, let $C_i:=\{n\in\Nat\colon X_n\text{ belongs to the }i\text{-th isometry class}\}$ and let $G:=\prod_{i\in I}S_{C_i}\leq S_\Nat$ be the corresponding closed subgroup of the full permutation group of $\Nat$ that setwise fixes each $C_i$. Then $\Iso(\bigoplus_{\ell_1} X_n)=G\ltimes \prod_{n\in\Nat} \Iso(X_n)$.
\end{proposition}
\begin{remark}
The semi-direct product $G\ltimes \prod_{n\in\Nat} \Iso(X_n)$ is defined with respect to the natural permutation action of $G$ on $\prod_{n\in\Nat} \Iso(X_n)$.
\end{remark}
\begin{proof}
	For a Banach space $X$, call a projection $P:X\to X$ an \emph{$L_1$-projection} if \[\|x\|=\|P(x)\|+\|x-P(x)\|,\quad x\in X\] and call it \emph{$L_\infty$-projection} if \[\|x\|=\max\big\{\|P(x)\|,\|x-P(x)\|\big\},\quad x\in X.\]
	
	The $L_1$-projections, resp. $L_\infty$ form a complete Boolean algebra, resp. a Boolean algebra (see \cite[Propositions 1.5 and 1.6]{Beh77}), denoted further by $\PP_1(X)$, resp. $\PP_\infty(X)$. Any surjective linear isometry $\phi:X\to X$ induces a complete Boolean algebra isomorphism $\phi_*:\PP_1(X)\to\PP_1(X)$ by the formula \[P\in\PP_1(X)\to \phi\circ P\circ \phi^{-1}\in\PP_1(X).\] Analogously, every surjective linear isometry $\phi:X\to X$ induces a Boolean algebra isomorphism $\phi_*:\PP_\infty(X)\to\PP_\infty(X)$.\medskip
	
	Suppose (1). If $(X_n)_{n\in\Nat}$ are as in the statement then for $\XX:=\bigoplus_{\ell_1} X_n$ we have that $\PP_1(\XX)$ is generated by the set of its atoms which consist of projections onto individual summands. If $\phi:\XX\to\XX$ is any surjective linear isometry, then $\phi_*:\PP_1(\XX)\to\PP_1(\XX)$ is a complete Boolean algebra isomorphism which therefore maps the atoms onto the atoms which further implies that $\phi$ maps individual summands onto individual summands. This proves (1).
	
	For (2) we proceed analogously. For $\XX:=\bigoplus_{c_0} X_n$ we have $\PP_\infty(\XX)$ is not complete, still its atoms are precisely the projections onto individual summands which are still preserved by every surjective linear isometry. We conclude as for (1).
\end{proof}

\begin{theorem}\label{thm:closurec0sums}
The GSCR property is closed under countable $c_0$-sums for spaces that do not admit decomposition as an $\ell_\infty$-sum of two proper subspaces.
\end{theorem}
Notice that this in particular applies to the case where each $X_n$ is strictly convex.
\begin{proof}
Set $\XX:=\bigoplus_{c_0} X_n$. We define an equivalence relation $\sim$ on $\Nat$ by declaring $i\sim j$ if $X_i$ is isometric to $X_j$. Let $I$ index the equivalence classes of $\sim$.

By the assumption, each summand $X_n$ admits an equivalent strictly convex norm $\|\cdot\|_n$ invariant with respect to $\Iso(X_n)$. Moreover, we can assume that for every $i\sim j$, $(X_i,\|\cdot\|_i)$ and $(X_j,\|\cdot\|_j)$ are isometric, witnessed by the same operator.

We define a new norm $\vertiii{\cdot}$ on $\XX$ as follows. For $(x_n)_n\in\XX$ set \[\vertiii{(x_n)_n}:=\|(x_n)_n\|_\XX+\Big(\sum_{i=1}^\infty \frac{\|x_i\|^2_i}{4^i}\Big)^{1/2}.\] Clearly, $\vertiii{\cdot}$ is equivalent and strictly convex. We claim that for every $(x_n)_n\in\XX$ the supremum \[\sup_{g\in\Iso(\XX)} \vertiii{g(x_n)_n}\] is attained. Then we will be done by Proposition~\ref{prop:iso-of-l1-c0-sums}. Pick any $C\in I$. We have $C\subseteq\Nat$ and we may write $C=\{c_n\colon n<N\}$, where $N\in\Nat\cup\{\Nat\}$ and $c_i<c_j$ for $i<j<N$.  Let $\pi_C:C\to C$ be a permutation satisfying $\|x_{\pi_C(c_i)}\|_{\pi_C(c_i)}\geq \|x_{\pi_C(c_j)}\|_{\pi_C(c_j)}$, for every $i<j<N$. Set $\pi:=\bigcup_{C\in I} \pi_C\in S_\Nat$ which is a permutation of $\Nat$. It is straightforward to check that the linear operator $h:\XX\to\XX$ that permutes the summands according to the permutation $\pi$ is a linear isometry, applying Proposition~\ref{prop:iso-of-l1-c0-sums}. Moreover, it satisfies $\vertiii{h(x_n)_n}=\sup_{g\in\Iso(\XX)}\vertiii{g(x_n)_n}$.
\end{proof}

Before proving an analogous result for $\ell_1$-sums, we stay with the spaces $c$, resp. $c_0$ a little bit longer and answer an open question from \cite{AFGR19}.

To introduce the question, we mention the result of Antunes, Ferenczi, Grivaux and Rosendal who showed (see \cite[Theorem 2.3 and Proposition 4.3]{AFGR19}) that the space $c$ does not admit an equivalent LUC norm invariant with respect to the canonical action of $\Iso(c)$. They asked, see \cite[Question 8.1]{AFGR19}, whether $c$ admits an equivalent strictly convex norm invariant with respect to this action of $\Iso(c)$. We show that this is indeed the case.
\begin{theorem}\label{thm:cGSCR}
The Banach space $c$ has GSCR.
\end{theorem}
In particular, there exists a Banach space that admits an equivalent strictly convex norm invariant with respect to the canonical action of its full linear isometry group, yet it does not admit an equivalent LUC norm invariant with respect to this action.

\begin{proof}
Since $\Iso(c)=S_\Nat\ltimes C\big(\{1/n\colon n\in\Nat\}\cup\{0\},\{-1,1\}\big)$ by the Banach-Stone theorem (see e.g. Section~\ref{sec:prelim} or \cite[Theorem 2.1.1]{FleJabook1}) viewing $c$ as $C\big(\{1/n\colon n\in\Nat\}\cup\{0\}\big)$, while $\Iso(c_0)=S_\Nat\ltimes \{-1,1\}^\Nat$ (see e.g. \cite[Proposition 2.5]{AB22}), we can view $\Iso(c)$ as a subgroup of $\Iso(c_0)$. In particular, $\Iso(c)$ acts canonically by linear isometries on $c_0$. Set \[H:=\{(\pi,f)\in S_\Nat\ltimes  C\big(\{1/n\colon n\in\Nat\}\cup\{0\},\{-1,1\}\big)\colon f(0)=1\},\] where by $(\pi,f)$ we denote the general form of an element of $\Iso(c)$, where $\pi\in S_\Nat$ and $f\in  C\big(\{1/n\colon n\in\Nat\}\cup\{0\},\{-1,1\}\big)$. Clearly, $H$ is an index $2$ subgroup of $\Iso(c)$.

By Corollary~\ref{cor:c0} there is a strictly convex norm $\vertiii{\cdot}$ on $c_0$ that is $\Iso(c)$-invariant. Notice that $\Iso(c)$ acts also by linear isometries on the strictly convex space $(c_0,\vertiii{\cdot})\oplus_{\ell_2}\Rea$ by acting canonically on the first summand and changing the sign of the second summand, $H$ being the kernel of this second action. The map \[(x_n)_{n\in\Nat}\to \big((x_n-\lim_{n\to\infty} x_n)\oplus \lim_{n\to \infty} x_n\big)\] is an $\Iso(c)$-equivariant bounded injective linear map from $c$ into $(c_0,\vertiii{\cdot})\oplus_{\ell_2}\Rea$. So we are done applying Proposition~\ref{prop:SCrenorming1}.
\end{proof}

Next we prove a preservation theorem for $\ell_1$-sums analogous to Theorem~\ref{thm:closurec0sums}.

\begin{theorem}\label{thm:l1sumspreservation}
The GSCR property is closed under countable $\ell_1$-sums for spaces that do not admit decomposition as an $\ell_1$-sum of two proper subspaces.
\end{theorem}

As was the case in Theorem~\ref{thm:closurec0sums}, this in particular applies to the case where each $X_n$ is strictly convex.

\begin{proof}
Set $\XX:=\bigoplus_{\ell_1} X_n$. We show there is a strictly convex renorming of $\XX$ invariant with respect to the canonical action of $\Iso(\XX)$.  Since for each $n\in\Nat$ there is a strictly convex renorming of $X_n$ with respect to the canonical action of $\Iso(X_n)$, by Proposition~\ref{prop:SCrenorming1} there exist a strictly convex Banach space $Y_n$ with an isometric action of $\Iso(X_n)$ and an $\Iso(X_n)$-equivariant injective linear operator $T_n:X_n\to Y_n$. We may suppose that $\|T_n\|\leq 1$.

Set $\YY:=\bigoplus_{\ell_2} Y_n$ which is strictly convex as an $\ell_2$-sum of strictly convex spaces. By Proposition~\ref{prop:iso-of-l1-c0-sums}, $\Iso(\XX)=G\ltimes\prod_{n\in\Nat} \Iso(X_n)$, where $G$ is as in the statement of Proposition~\ref{prop:iso-of-l1-c0-sums}. Thus there is a natural action of $\Iso(\XX)$ on $\YY$. That is, the natural action of $G\ltimes\prod_{n\in\Nat} \Iso(X_n)$ where the subgroup $G$ permutes the sum $\bigoplus_{\ell_2} Y_n$ and $\prod_{n\in\Nat} \Iso(X_n)$ acts coordinatewise. The map \[T:\XX\to\YY,\quad (x_n)_{n\in\Nat}\to (T_n(x_n))_{n\in\Nat}\] is then an injective $\Iso(\XX)$-equivariant linear operator. Indeed, it is well-defined since given $(x_n)_{n\in\Nat}\in\XX$ we have $\sum_{i=1}^\infty \|x_n\|_{X_n}<\infty$, thus, since $\|T_n\|\leq 1$ for each $n$, $\sum_{i=1}^\infty \|T_n(x_n)\|_{Y_n}^2<\infty$. Clearly, it is linear and $\Iso(\XX)$-equivariant. Injectivity follows from the injectivity of each $T_n$, $n\in\Nat$. Then we are done by Proposition~\ref{prop:SCrenorming1}.
\end{proof}

The following is an immediate corollary which however already follows from Theorem~\ref{thm:Lancien} since $\ell_1$ has the Radon-Nikod\'ym property.
\begin{corollary}
Let $G$ be a topological group acting by isometries on $\ell_1$. Then there exists a strictly convex $G$-invariant norm on $\ell_1$.
\end{corollary}
\begin{proof}
Since $\ell_1=\bigoplus_{\ell_1} \Rea$ this is immediate from Theorem~\ref{thm:l1sumspreservation}.
\end{proof}

\section{Actions of countable groups on $L_1[0,1]$}\label{sec:L1}
Not every separable Banach space has GSCR. It was shown in \cite{AFGR19} (see also \cite[Example 38]{FR11}) that there is no strictly convex renorming of $L_1[0,1]$ invariant with respect to the canonical action of $\Iso(L_1[0,1])$. Here we show that one may actually find an action of a free group on $2$ generators on $L_1[0,1]$ for which there is no invariant strictly convex renorming. We were informed by Mikael de la Salle that a similar result was privately communicated to him by Valentin Ferenczi.
\begin{proposition}\label{prop:nastyF2}
There exists an (lattice) isometric action $F_2\curvearrowright L_1[0,1]$ of the free group on two generators that does not admit an invariant strictly convex renorming. 
\end{proposition}
\begin{proof}
We define the following maps $\phi:[0,1]\to [0,1]$ and $h:[0,1]\to \{\frac{1}{2},2\}$. For $x\in[0,1]$ set \[\phi(x):=\begin{cases} 2x & x\in [0,\frac{1}{3}],\\
			\frac{x-\frac{1}{3}}{2}+\frac{2}{3} & x\in[\frac{1}{3},1],
\end{cases}\]
and
\[h(x):=\begin{cases}\frac{1}{2} & x\in [0,\frac{1}{3}],\\
2 & x\in [\frac{1}{3},1].
	
\end{cases}\]

Now we define $T_1:L_1[0,1]\to L_1[0,1]$ by the following formula. Given $f\in L_1[0,1]$ and $x\in [0,1]$ we set \[T(f)(x):=h\big(\phi^{-1}(x)\big)\cdot f\big(\phi^{-1}(x)\big).\] It is straightoforward to verify that $T_1$ is a linear isometry. Notice moreover that $T(\chi_{[0,1/3]})=1/2 \chi_{[0,2/3]}$.

Next we define $\psi:[0,1]\to[0,1]$ to be any measure preserving map of the unit interval such that $\psi([0,\frac{1}{3}])=[\frac{1}{3},\frac{2}{3}]$, $\psi([\frac{1}{3},\frac{2}{3}])=[\frac{2}{3},1]$ and $\psi([\frac{2}{3},1])=[0,\frac{1}{3}]$. Then we define $T_2:L_1[0,1]\to L_1[0,1]$ by the following formula. Given $f\in L_1[0,1]$ and $x\in [0,1]$ we set \[T_2(f)(x):=f\big(\psi^{-1}(x)\big).\] It is clear that $T_2$ is a linear isometry and clearly $T_2(\chi_{[0,1/3]})=\chi_{[1/3,2/3]}$.\medskip

Define now an action of a free group on two generators where the first generator acts as $T_1$ and the other as $T_2$. Suppose that there is an equivalent strictly convex norm $\|\cdot\|'$ of $L_1[0,1]$ invariant with respect to this $F_2$ action. Then on one hand we have \[\|\chi_{[0,1/3]}\|'=\|T_2(\chi_{[0,1/3]})\|'=\|\chi_{[1/3,2/3]}\|'.\] On the other hand we have \[\|\chi_{[0,1/3]}\|'=\|T_1(\chi_{[0,1/3]})\|'=\|1/2\chi_{[0,2/3]}\|'.\] However, since $\|\cdot\|'$ is strictly convex we must have \[\|1/2\chi_{[0,2/3]}\|'<1/2\|\chi_{[0,1/3]}\|'+1/2\|\chi_{[1/3,2/3]}\|'=\|\chi_{[0,1/3]}\|'\] which is a contradiction.
\end{proof}

We remark that all the counterexamples of separable Banach spaces equipped with an action of a (topological) group without an equivalent invariant strictly convex renorming, that we know, are based on actions of the free group on two generators. See Section~\ref{sec:C(K)} for further results and \cite[Example 5.3]{AFGR19} for an $F_2$ action on the non-separable $\ell_\infty$ without an invariant strictly convex renorming (note that there is a strictly convex renorming of $\ell_\infty$).

In the rest of this section, we would like to understand the existence of invariant strictly convex renormings of $L_1[0,1]$, under the presence of an action of a countable group, further. Since the culprit in the counterexamples is always $F_2$, the free group on two generators, we are interested in whether this group has (faithful) actions that admit invariant strictly convex renormings and whether the situation changes when we act with countable amenable groups.

Recall from Section~\ref{sec:prelim} that $\LIso(L_1[0,1])$ can be identified with $\Aut([0,1],\lambda)$ and thus actions of countable groups on $L_1[0,1]$ by positive linear isometries correspond to action on $[0,1]$ by measure-class preserving measurable bijections.
\begin{proposition}\label{prop:domainbijections}
	Let $\alpha$ be an action of a countable group $G$ on $L_1[0,1]$ by lattice isometries. Suppose that the induced continuous action of $G$ on $[0,1]$ by measure class preserving bijections admits a measurable fundamental domain, i.e. a measurable set $A\subseteq [0,1]$ such that $\lambda(gA\cap hA)=0$ for $g\neq h\in G$ and $\lambda\big([0,1]\triangle \bigcup_{g\in G} gA\big)=0$. Then there is a strictly convex $G$-invariant renorming of $L_1[0,1]$.
\end{proposition}
\begin{proof}
	For every $g\in G$, the pushforward $g_*\lambda$ of the Lebesgue measure shares the same null sets with $\lambda$, thus there is a a.e. positive function $T_g:[0,1]\to\Rea$, the Radon-Nikodym derivative of $g_*\lambda$ with respect to $\lambda$, such that for every measurable $C\subseteq [0,1]$, $g_*\lambda(C)=\int_C T_g d\lambda$. For every $g\in G$, $f\in L_1[0,1]$ and $x\in [0,1]$ we then have $\alpha(g)f(x)=T_g(x)f(g^{-1}x)$.
	
	For every $n\in \Nat$, let $A_1^n,A_2^n,\ldots,A_n^n$ be a pairwise disjoint subsets of the measurable fundamental domain $A$ such that for $i\leq n$, $\lambda(A_i^n)=\lambda(A)/n$. Now for each $n\in\Nat$ write $G\times n$ as a shortcut for $G\times\{1,\ldots,n\}$ and define a map $P_n: L_1[0,1]\to \ell_1(G\times n)$ by \[f\in L_1[0,1]\to \Big((g,i)\to \int_{gA_i^n} f d\lambda\Big),\quad g\in G, i\leq n.\] $P_n$ is clearly a norm-one linear operator. We claim that it is $G$-equivariant where the action of $G$ on $\ell_1(G\times n)$ is given by the action of $G$ on $G\times n$ acting by left-translation on the first coordinate and trivially on the second coordinate. Indeed, for $f\in L_1[0,1]$, $g,h\in G$ and $i\leq n$ we have \[\begin{split}P_n(\alpha(h)f)(g,i)& =\int_{gA_i^n} \alpha(h)f d\lambda=\int_{gA_i^n} T_hf(h^{-1}\cdot)d\lambda = \int_{gA_i^n} f(h^{-1}\cdot)dh_*\lambda\\ &=\int_{h^{-1}gA_i^n} f d\lambda=P_n(f)(h^{-1}g,i)=hP_n(f)(g,i).\end{split}\]
	
	It is standard that for every $f\in L_1[0,1]$ we have \[\lim_{n\to \infty}\big|\|f\|_{L_1[0,1]}-\|P_n(f)\|_{\ell_1(G,n)}\big|=0.\] Denoting the inclusion map $\ell_1(G,n)\to \ell_p(G,n)$ by $T^p_n$, for $n\in\Nat$ and $p\in (1,\infty)$, we have that $T^p_n$ is a norm one $G$-equivariant linear operator. It follows that for every $f\in L_1[0,1]$ there are $n\in\Nat$ and $p\in (1,\infty)$ such that $T^p_n\circ P_n$ is a norm one linear $G$-equivariant operator into a strictly convex space satisfying $\|T^p_n\circ P_n(f)\|\geq 2\|f\|$. Applying this to a countable dense subset $(f_n)_{n\in\Nat}\subseteq L_1[0,1]$, we are done by Proposition~\ref{prop:SCrenorming2}.
\end{proof}

As a corollary we get that every countable group, including $F_2$ which by Proposition~\ref{prop:nastyF2} admits an action without an invariant strictly convex renorming, can act nicely on $L_1[0,1]$.
\begin{corollary}\label{cor:actiononL1}
Every countable group $G$ admits a free action by lattice isometries on $L_1[01]$ that has an invariant strictly convex renorming.
\end{corollary}
\begin{proof}
By Proposition~\ref{prop:domainbijections}, it suffices to produce for every countable $G$ a free measure class preserving action on $[0,1]$ with a measurable fundamental domain. This is however straightforward. Split $[0,1]$ into countably many closed intervals, where the intersection of pairwise different intervals has measure zero (and their union covers $[0,1]$) and index them by elements of $G$, i.e. as $(I_g)_{g\in G}$. There is an obvious set-action of $G$ on $(I_g)_{g\in G}$ defined by $gI_h=I_{gh}$ which can be extended in a straightforward way to a measure-class preserving point-action on $[0,1]$.
\end{proof}

We are not able to decide whether every action of a countable \emph{amenable} group on $L_1[0,1]$ by linear or positive linear isometries admits an invariant strictly convex renorming. However, we can prove that such actions are dense in $\Act_G^+(L_1[0,1])$ (recall the definition from Section~\ref{sec:prelim}).

Additionally, we denote by $\Act_G([0,1],\lambda)$ the space of all actions of $G$ on $([0,1],\lambda)$ by measure-class preserving measurable bijections which can be identified with a closed subset of $\Aut([0,1],\lambda)^G$ and is canonically homeomorphic to $\Act^+_G(L_1[0,1])$ as $\Aut([0,1],\lambda)$ is topologically isomorphic to $\LIso(l_1[0,1])$ (recall that from Section~\ref{sec:prelim}).

We shall need the following lemma. Recall that an action $\alpha\in\Act_G([0,1],\lambda)$ is \emph{essentially free} if for every $g\neq 1_G\in G$, the set $\{x\in[0,1]\colon \alpha(g)x=x\}$ has measure zero.
\begin{lemma}\label{lem:essfree}
Given a countable group $G$, the set of essentially free actions is dense in $\Act_G([0,1],\lambda)$.
\end{lemma}
The lemma is well-known in the more specialized case of measure preserving actions, see e.g. \cite[Theorem 10.8]{Kech10}. Since the proof in the more general case is essentially the same, we provide a sketch. We refer to the above references for more details.
\begin{proof}[Sketch of the proof.]
Since $\Act_G([0,1],\lambda)$ is separable, there exists a countable dense sequence $(\alpha_n)_{n\in\Nat}\subseteq\Act_G([0,1],\lambda)$. Without loss of generality, we may assume that $\alpha_1$ is essentially free. Such actions exist, see e.g. \cite[Theorem 10.8]{Kech10}. Now consider $\prod_{n\in\Nat} \alpha_n$ as a diagonal action of $G$ on $([0,1]^\Nat,\lambda^\Nat)$. Since $\alpha_1$ is essentially free, $\prod_{n\in\Nat} \alpha_n$ is essentially free as well. Moreover, since there is a measurable measure-preserving bijection between $([0,1],\lambda)$ and $([0,1]^\Nat,\lambda^\Nat)$, conjugating $\prod_{n\in\Nat} \alpha_n$ with this bijection, we may view it as an element of $\Act_G([0,1],\lambda)$. By the same argument as in \cite[Theorem 10.8]{Kech10}, the conjugacy class of $\prod_{n\in\Nat} \alpha_n$ is dense in $\Act_G([0,1],\lambda)$, thus in particular, essentially free actions are dense.
\end{proof}

\begin{theorem}\label{thm:densityforL1}
Let $G$ be a countable amenable group. Then the set of lattice isometric actions of $G$ on $L_1[0,1]$ that admit a strictly convex renorming is dense in $\Act^+_G(L_1[0,1])$.
\end{theorem}
\begin{proof}
Let $U$ be an open set in $\Act^+_G(L_1[0,1])$. We shall find an action $\beta\in U$ such that there is a strictly convex renorming of $L_1[0,1]$ invariant with respect to the action $\beta$. We may suppose $U$ is a basic open set of the form \[\{\beta:G\to\LIso(L_1[0,1])\colon \forall i\leq n\;\forall j\leq m\; \big(\|\beta(g_i)f_j-\alpha(g_i)f_j\|<\varepsilon\big)\},\] where $\alpha$ is some fixed action of $G$, $g_1,\ldots,g_n\in G$, $f_1,\ldots,f_m\in B_{L_1[0,1]}$, and $\varepsilon>0$.

By Lemma~\ref{lem:essfree}, without loss of generality we may assume that the measure class preserving action of $G$, induced by $\alpha$, on $[0,1]$ is essentially free. That is, for each $g\in G$, we have \[\lambda\Big(\big\{x\in [0,1]\colon gx=x\big\}\Big)=0.\] Set $K:=\{g_1,\ldots,g_n\}$. Moreover, without loss of generality, we may assume that $K$ is symmetric and $1_G\in K$.

Fix $0<\delta<\varepsilon^3/250$. By combining \cite[Theorem 4.3]{DoHuZh19} and \cite[Theorem 2.16]{Sci24}, which jointly give a strong generalization of the Ornstein-Weiss thereom from \cite{OrWei87} (which itself is a powerful generalization of the classical Rokhlin's lemma from ergodic theory),  there is a finite collection of finite tiles $(T_i)_{i\leq k}$ (we refer to \cite[Section 3]{DoHuZh19} for a precise definition, here we only see them as finite sets) of $G$ and finitely many measurable subsets $(W_i)_{i\leq k}$ of $[0,1]$ such that 
\begin{itemize}
	\item for $t\in T_i$, $s\in T_j$, $i\neq j\leq k$, we have $tW_i\cap sW_j=\emptyset$;
	\item for every $i\leq k$, $T_i$ is $(K,\delta)$-invariant, i.e. for every $g\in K$ \[\frac{|gT_i\triangle T_i|}{|T_i|}<\delta;\]
	\item $\lambda(\bigcup_{i\leq k} T_iW_i)>1-\varepsilon/4$;
	\item $\sum_{i=1}^k \lambda\big((\partial_K T_i)W_i\big)<\varepsilon/4$.
\end{itemize}
Without loss of generality, we may suppose that $1_G\in T_1\setminus\partial_K T_1$, $\lambda(W_i)>0$, for every $i\leq k$, and that $\lambda([0,1]\setminus\bigcup_{i\leq k} T_iW_i)>0$. 


Now for each $i\leq k$ and $g\in T_i$, set $A_g:=gW_i$. Notice that for $g\neq h\in \bigcup_{i\leq k} T_i$, the sets $A_g$ and $A_h$ have positive measure and are disjoint. Since by the assumption \[\lambda\Big([0,1]\setminus \bigcup_{g\in \bigcup_{i\leq k} T_i} A_g\Big)>0\]
we can also define $A_h$ for every $h\notin \bigcup_{i\leq k} T_i$ so that it has positive measure and so that the sets $(A_g)_{g\in G}$ are pairwise disjoint and $\bigcup_{g\in G} A_g=[0,1]$.

Now pick any measure-class preserving action of $G$ on $[0,1]$ satisfying $gA_h=A_{gh}$, for every $g,h\in G$, that coincides with the original action on $\bigcup_{i\leq k} (T_i\setminus \partial_K T_i)W_i$. Denote by $\beta$ the corresponding lattice isometric action on $L_1[0,1]$. We claim that $\beta\in U$. Pick any $g\in K$ and $j\leq m$ and we show that $\|\beta(g)f_j-\alpha(g)f_j\|<\varepsilon$. 

We have \[\begin{split}\|\beta(g)f_j-\alpha(g)f_j\| &=\int_{[0,1]} |\beta(g)f_j(x)-\alpha(g)f_j(x)|d\lambda(x)\\ &=\int_{\bigcup_{i\leq k} (T_i\setminus\partial_K T_i)W_i} |\beta(g)f_j(x)-\alpha(g)f_j(x)|d\lambda(x)\\ &+\int_{[0,1]\setminus \bigcup_{i\leq k} (T_i\setminus\partial_K T_i)W_i} |\beta(g)f_j(x)-\alpha(g)f_j(x)|d\lambda(x)\end{split}\]

Since by definition, the measure class preserving actions of $G$ on $[0,1]$ induced by $\alpha$ and $\beta$ respectively coincide on $\bigcup_{i\leq k} (T_i\setminus\partial_K T_i)W_i$ we get \[\int_{\bigcup_{i\leq k} (T_i\setminus\partial_K T_i)W_i} |\beta(g)f_j(x)-\alpha(g)f_j(x)|d\lambda(x)=0.\] On the other hand, since \[\lambda\big([0,1]\setminus \bigcup_{i\leq k} (T_i\setminus\partial_K T_i)W_i\big)=\Big(\sum_{i=1}^k \lambda\big((\partial_K T_i)W_i\big)\Big)+\lambda\big([0,1]\setminus\bigcup_{i\leq k} T_iW_i\big)<\varepsilon/2\] we get 

\[\begin{split}
& \int_{[0,1]\setminus \bigcup_{i\leq k} (T_i\setminus\partial_K T_i)W_i} |\beta(g)f_j(x)-\alpha(g)f_j(x)|d\lambda(x)\\ & \leq \int_{[0,1]\setminus \bigcup_{i\leq k} (T_i\setminus\partial_K T_i)W_i} |\beta(g)f_j(x)|d\lambda(x)\\&+ \int_{[0,1]\setminus \bigcup_{i\leq k} (T_i\setminus\partial_K T_i)W_i} |\alpha(g)f_j(x)|d\lambda(x)\\ &\leq 2\lambda\big([0,1]\setminus \bigcup_{i\leq k} (T_i\setminus\partial_K T_i)W_i\big)<\varepsilon.
\end{split}\]

Thus we have $\|\beta(g)f_j-\alpha(g)f_j\|<\varepsilon$ as desired. Finally, applying Proposition~\ref{prop:domainbijections}, we get that there is a strictly convex renorming of $L_1[0,1]$ invariant with respect to $\beta$.

\end{proof}

\section{Actions on $C(K)$-spaces}\label{sec:C(K)}
Finally, we investigate invariant strictly convex renormings on spaces of the form $C(K)$, where $K$ is a compact Hausdorff space. Besides several general results showing the significance of amenability of groups in this setting, we demonstrate that the problem whether the `majority' of the actions admit an invariant strictly convex renorming or not strongly depends on the acting group.

Recall from Section~\ref{sec:prelim} that given a group $G$, a continuous action $\alpha$ of $G$ on $C(K)$ by linear isometries can be identified with a continuous homomorphism from $G$ to $\Iso(C(K))=\Homeo(K)\ltimes C(K,\{-1,1\})$. Composing this homomorphism with the quotient map $p:\Iso(C(K))=\Homeo(K)\ltimes C(K,\{-1,1\})\to \Homeo(K)$, we obtain a continuous homomorphism from $G$ to $\Homeo(K)$ which can be identified with a continuous action of $G$ on $K$, which we shall denote by $\alpha'$.
\begin{theorem}\label{thm:cptspaces}
Let $G$ be an amenable group acting by isometries on $C(K)$, where $K$ is a compact Hausdorff space, and let $\alpha':G\curvearrowright K$ be the induced continuous action of $G$ on $K$.
\begin{enumerate}
	\item If either $K$ can be written as a finite disjoint union of minimal subspaces of $\alpha'$,
	\item or $K$ is metrizable and can be written as a disjoint union of (arbitrarily many) minimal subspaces of $\alpha'$
\end{enumerate}
then $C(K)$ admits a $G$-invariant strictly convex renorming.
\end{theorem}
\begin{proof}
Suppose (1) and write K=$\coprod_{i\leq n} K_i$, where each $K_i$ is a minimal subspace for the action of $G$. By amenability, each $K_i$ admits a Borel probability $G$-invariant measure $\mu_i$. Notice that the support of $\mu_i$ is the whole $K_i$ since the support is a non-empty closed $G$-invariant subset. Therefore, $\mu:=\sum_{i=1}^n \mu_i$ is a finite Borel fully supported $G$-invariant measure on $K$.

It follows that the inclusion $C(K)\subseteq L_\infty(K,\mu)$ is a linear isometric $G$-equivariant operator. Thus it suffices to show that $L_\infty(K,\mu)$ has a strictly convex $G$-invariant renorming. However, this follows since the inclusion $L_\infty(K,\mu)\subseteq L_2(K,\mu)$ is a continuous $G$-equivariant linear operator and $L_2(K,\mu)$ is strictly convex. So we are done by Proposition~\ref{prop:SCrenorming1}.\medskip

Suppose (2). Fix a countable dense sequence $(f_n)_{n\in\Nat}\subseteq C(K)$. We show that for every $n\in\Nat$ there is a $G$-equivariant bounded linear map $T_n:C(K)\to X_n$, where $X_n$ is strictly convex, $\|T_n\|\leq 1$ and $\|T_n(f_n)\|\geq \|f_n\|/4$. Then we will be done by the application of Proposition~\ref{prop:SCrenorming2}.

Fix $n\in\Nat$ and let $x\in K$ be such that $|f_n(x)|=\|f_n\|$. By the assumption, $K_x:=\overline{G\cdot x}$ is a closed minimal subspace of $K$. Since $G$ is amenable, there is a Borel probability $G$-invariant measure $\mu_x$ on $K_x$. The inclusion operator $E_n:C(K)\to L_\infty(K,\mu_x)$ satisfies $\|f_n\|=\|E_n(f_n)\|$ and $\|E_n\|\leq 1$. Notice that $E_n$ is $G$-equivariant. On the other hand, since $\mu_x$ is finite, the $G$-equivariant inclusion operator $L_\infty(K,\mu_x)\to L_2(K,\mu_x)$ shows that $L_\infty(K,\mu_x)$ admits a $G$-invariant strictly convex renorming by the application of Proposition~\ref{prop:SCrenorming1}. However, then by Corollary~\ref{cor:SCrenorming}, $L_\infty(K,\mu_x)$ admits a $G$-invariant strictly convex norm $\vertiii{\cdot}$ such that $1/2\|\cdot\|_\infty\leq \vertiii{\cdot}\leq 2\|\cdot\|_\infty$. It follows that for the inclusion operator $T'_n:C(K)\to (L_\infty(K,\mu_x),\vertiii{\cdot})$ we have $\|T'_n\|\leq 2$ and $\vertiii{T'_n(f_n)}\geq 1/2\|f_n\|$. Therefore, we may set $T_n:=T'_n/2$ and we are done.
\end{proof}
\begin{remark}
Notice that the theorem is a generalization of the fact that each separable $C(K)$ space admits a strictly convex renorming. Indeed, in that case, the trivial group is amenable and $K$ is a disjoint union of its singletongs which are minimal subspaces for the action of the trivial group.
\end{remark}

Let $K$ be again a compact Hausdorff space and $G$ be a countable group. Analogously as in Section~\ref{sec:L1}, we aim to investigate the spaces $\Act_G(C(K))$, resp. $\Act_G^+(C(K))$.

We shall also denote by $\Act_G(K)$ the space of all homomorphisms of $G$ into the group $\Homeo(K)$, which can be viewed as a closed subset of $\Homeo(K)^G$ and further identified with the space of all actions of $G$ on $K$ by homeomorphisms. It is naturally homeomorphic with $\Act_G^+(C(K))$.

Note that the quotient map $q:\Iso(C(K))\to\Homeo(K)$ is continuous and open. It follows that by post-composition, for every countable group $G$, $q$ induces a map $\tilde{q}:\Act_G(C(K))\to \Act_G(K)$ which is also continuous and open. Therefore the following fact is clear.
\begin{fact}\label{fact:densepreimage}
Let $K$ be a compact Haudorff space and $G$ be a countable group. If $F\subseteq \Act_G(K)$ is dense, then the preimage $\tilde q^{-1}(D)\subseteq\Act_G(C(K))$ is dense as well.
\end{fact}

Recall that a group is \emph{locally finite} if every finitely generated subgroup is finite. Such groups are in particular amenable and majority of their actions on $C(2^\Nat)$ by linear isometries admit an invariant strictly convex renorming as demonstrated by the following result. This will be in stark contrast to Theorem~\ref{thm:haveF2}.
\begin{theorem}\label{thm:locfinite}
Let $G$ be a countable locally finite group. Then the set of linear isometric actions of $G$ on $C(2^\Nat)$ that admits a strictly convex renorming is comeager in $\Act_G(C(2^\Nat))$.
\end{theorem}
\begin{proof}
Notice that if $G$ is finite, then it is in particular compact and so by Corollary~\ref{cor:c0} every linear isometric action of $G$ admits a strictly convex renorming. So we may assume that $G$ is infinite. By \cite[Theorem 7.11]{DoMeTs25}, there is a dense $G_\delta$ set $D\subseteq \Act_G(2^\Nat)$ consisting of minimal uniquely ergodic actions. The preimage $\tilde{D}:=\tilde{q}^{-1}(D)\subseteq \Act_G(C(2^\Nat))$ is $G_\delta$ by continuity of $\tilde{q}$ and dense by opennes of $\tilde{q}$ (see Fact~\ref{fact:densepreimage}). So it suffices to check that every action from $\tilde{D}$ admits a strictly convex renorming. However, this follows from Theorem~\ref{thm:cptspaces}.
\end{proof}

Notice that given a countable group $G$ and a compact Hausdorff space $K$, there is a natural action of $\LIso(C(K))$ on $\Act^+_G(C(K))$ by conjugation  denoted by $\alpha\in\Act^+_G(C(K))\to \phi\alpha\phi^{-1}$, for $\phi\in\LIso(C(K))$. The conjugation is defined as follows. For $\alpha\in\Act^+_G(C(K))$, $\phi\in\LIso(C(K))$ we have \[\big(\phi\alpha\phi^{-1}\big)(g):=\phi\alpha(g)\phi^{-1}\in\LIso(C(K)),\;\; g\in G.\]

Analogously, there is an action by conjugation of $\Homeo(2^\Nat)\leq \Iso(C(2^\Nat))$ on $\Act_G(2^\Nat)$.
\begin{theorem}\label{thm:haveF2}
Let $G$ be a countable group that contains $F_2$, the free group on two generators, as a subgroup. Then the set of positive isometric actions of $G$ on $C(2^\Nat)$ that does not admit a strictly convex renorming is comeager in $\Act^+_G(C(2^\Nat))$.
\end{theorem}

\begin{proof}

Given $\alpha\in\Act_G^+(C(2^\Nat))$ we shall denote by $\alpha'\in\Act_G(2^\Nat)$ the induced action of $G$ on $2^\Nat$ by homeomorphisms.

Let $g,h\in G$ be two elements that freely generate a copy of $F_2$ inside $G$.  It suffices to check that the set $\DD$ of those $\alpha\in\Act^+_G(C(2^\Nat))$ such that there exists a clopen partition of $2^\Nat$ into non-empty clopen sets $A$, $B$, $C$, $D$, and $E$ such that\bigskip
\begin{itemize}
	\item $\alpha'(g)[A]=A\cup B\cup C$;
	\item $\alpha'(g)[B\cup D]=D$ and $\alpha'(g)[C\cup E]=E$;
	\item $\alpha'(h)[F]=F$ for $F\in\{A,D,E\}$, and $\alpha'(h)[B]=C$, $\alpha'(h)[C]=B$;
	
	\medskip
\end{itemize}

is comeager in $\Act^+_G(C(2^\Nat))$.\medskip

Indeed, let $\alpha\in\DD$ which is witnessed by a clopen partition of $2^\Nat$ into non-empty clopen sets $A$, $B$, $C$, $D$, and $E$. Set $f\in C(2^\Nat)$ to be $1/2\chi_A+\chi_B+\chi_D$ and $f'\in C(2^\Nat)$ to be $1/2\chi_A+\chi_C+\chi_D$. Let $\vertiii{\cdot}$ be any equivalent norm on $C(2^\Nat)$ invariant with respect to $\alpha$. Since $\alpha(h)f= f'$ we get $\vertiii{f}=\vertiii{f'}$. On the other hand, we have $\alpha(g)f=\frac{f+f'}{2}$, so $\vertiii{\frac{f+f'}{2}}=\vertiii{f}=\vertiii{f'}$. Therefore, $\vertiii{\cdot}$ is not strictly convex.\bigskip

Let us show that $\DD$ is comeager. In fact, we show that it is open and dense. First observe that $\DD$ is open. Next we show it is dense. First recall that a continuous action of a topological group $H$ on a topological space $S$ is \emph{topologically transitive} if for every open sets $U,V\subseteq S$ there exists $g\in G$ such that $gU\cap V\neq\emptyset$. Since the action of $\Homeo(2^\Nat)$ on $\Act_G(2^\Nat)$ by conjugation is topologically transitive (see e.g. \cite[Proposition 1.2]{Hoch12} stated for $\Int^d$ but working for every countable group), the action of $\LIso(C(2^\Nat))$ on $\Act_G^+(C(2^\Nat))$ by conjugation is topologically transitive as well. It follows that for every open set $U\subseteq \Act_G^+(C(2^\Nat))$, if $\DD\neq\emptyset$, there is $\phi\in \LIso(C(2^\Nat))$ such that $\phi\DD\phi^{-1}\cap U\neq\emptyset$. However, since $\DD$ is easily checked to be conjugation-invariant, it immediately follows that $\DD$ is dense if and only if it is non-empty.

Let us therefore show that $\DD$ is non-empty. Define $X\subseteq S^\Int$, where $S:=\{-2,-1,0,1,2\}$ to be a set consisting of elements $x\in S^\Int$ such that:
\begin{itemize}
	\item either there exists $n\in\Int$ such that $x(n)=0$, $x(m)=1$ for all $m>n$ and $x(m')=-1$ for all $m'<n$;
	\item or there exists $n\in\Int$ such that $x(n)=0$, $x(m)=2$ for all $m> n$ and $x(m')=-2$ for all $m'<n$;
	\item or $x$ is constantly equal to $-2$ or $-1$ or $1$ or $2$.
\end{itemize}
Clearly, $X$ is a closed subset of $S^\Int$ closed under the shift $\sigma:S^\Int\to S^\Int$ defined by $\sigma(x)(n):=x(n-1)$, for $x\in S^\Int$ and $n\in\Int$. Notice that $X$ is totally disconnected and set 
\begin{itemize}
	\item $A':=\{x\in X\colon x(0)=x(1)=1\}\cup \{x\in X\colon x(0)=x(1)=2\}$;
	\item $B':=\{x\in X\colon x(0)=0, x(1)=1\}$;
	\item $C':=\{x\in X\colon x(0)=0, x(1)=2\}$;
	\item $D':=\{x\in X\colon x(0)=-1, x(1)=0\}\cup \{x\in X\colon x(0)=x(1)=-1\}$;
	\item $E':=\{x\in X\colon x(0)=-2, x(1)=0\}\cup \{x\in X\colon x(0)=x(1)=-2\}$.
\end{itemize}
Notice that $A'$, $B'$, $C'$, $D'$, $E'$ form a clopen partition of $X$. Now set $Y:=X\times 2^\Nat$. Clearly, $Y$ is homeomorphic to $2^\Nat$. For $F'\in\{A',B',C',D',E'\}$ set $F:=F'\times 2^\Nat$. It follows that $A$, $B$, $C$, $D$, $E$ is a clopen partition of $Y$. Define a homeomorphism $T:Y\to Y$ by $\sigma\times \mathrm{id}$, i.e. it acts as the shift $\sigma$ on the first coordinate $X$ and as the identity on the second coordinate $2^\Nat$. Notice that we have $T[A]=A\cup B\cup C$, $T[B\cup D]=D$ and $T[C\cup E]=E$. Since all non-empty clopen subsets of the Cantor space are homeomorphic, there clearly exists a homeomorphism $T':Y\to Y$ such that $T'[F]=F$ for $F\in\{A,D,E\}$ and $T'[B]=C$, $T'[C]=B$. We can define an action of $F_2$ on $2^\Nat$ so that the first generator, $g$, acts as $T$, and the second generator, $h$, acts as $h'$. Using co-induction (see e.g. \cite[Section 7.1]{DoMeTs25}), we can extend this action of $F_2$ on $2^\Nat$ to the action of the whole group $G$ on $2^\Nat$. Denote by $\alpha$ the corresponding action of $G$ by positive linear isometries on $C(2^\Nat)$. Clearly, $\alpha\in\DD$ with respect to the clopen sets $A,B,C,D,E$ defined above. This shows that $\DD$ is non-empty and finishes the proof.
\end{proof}

\begin{remark}
Theorem~\ref{thm:haveF2} is in stark contrast to Theorem~\ref{thm:locfinite}. The classes of locally finite groups and groups having $F_2$ as a subgroup are obviously disjoint, however they do not cover the set of all countable groups - even $\Int$ is not covered by either of them. It is tempting to believe that the dividing line separating locally finite groups and groups having $F_2$ as a subgroup is precisely the set of countable amenable groups and precisely for these groups, a comeager set of actions admit an invariant strictly convex renorming.
\end{remark}

\section{Some open problems}\label{sec:questions}
Finally, we collect several open problems that we consider interesting and important and whose solutions could push the topic further.

First, the only examples of separable $G$-Banach spaces not admitting an equivalent $G$-invariant strictly convex norms are based on actions of free groups on two generators. The following is therefore pressing.

\begin{question}
Do there exist a separable Banach space $X$ and a surjective linear isometry $T:X\to X$ such that there is no $T$-invariant strictly convex renorming of $X$?
\end{question}
Note that this corresponds to asking whether every action of $\Int$ on a separable Banach space admits a $\Int$-invariant strictly convex renorming.

In case the answer is positive, it makes sense to ask the question in bigger generality.

\begin{question}
Let $G$ be a countable amenable group and $X$ be a separable $G$-Banach space. Does there exist a $G$-invariant strictly convex renorming of $X$?
\end{question}

It was pointed out to us by Christian Rosendal that without the assumption of countability of $G$, the previous question has a negative answer. Indeed, the group $\Iso(L_1[0,1])$ is (extremely) amenable by \cite[Theorem 6.6]{GioPe07}, while $L_1[0,1]$ does not admit an $\Iso(L_1[0,1])$-invariant strictly convex renorming by \cite[Example 38]{FR11}.\medskip

The spaces $L_1[0,1]$ and $C(2^\Nat)$ considered in Sections~\ref{sec:L1} and~\ref{sec:C(K)} respectively are useful test spaces as their full isometry groups are rather well-understood and approachable. The following two problems aim to generalize the main results from Sections~\ref{sec:L1} and~\ref{sec:C(K)} respectively.

\begin{question}
Let $G$ be a countable group. Is the set of actions admitting, resp. not admitting, a $G$-invariant strictly convex renorming comeager in $\Act_G(L_1[0,1])$ or $\Act_G^+(L_1[0,1])$ if and only if $G$ is amenable, resp. non-amenable?
\end{question}

\begin{question}
Let $G$ be a countable group. Is the set of actions admitting, resp. not admitting, a $G$-invariant strictly convex renorming comeager in $\Act_G(C(2^\Nat))$ or $\Act_G^+(C(2^\Nat))$ if and only if $G$ is amenable, resp. non-amenable?
\end{question}\bigskip

\noindent\textbf{Acknowledgements.} We would like to thank to Mikael de la Salle for asking the question that motivated this work and sharing his notes with us. We also thank to Christian Rosendal for his comments on the role of amenability in the subject of the paper, to Yves Raynaud for the reference \cite{Beh77} and to Julien Melleray for helping to find a reference for Lemma~\ref{lem:essfree}.
\bibliographystyle{siam}
\bibliography{references-renormings}
\end{document}